\documentclass{article}
\usepackage{hyperref}
\usepackage[english]{babel}

\usepackage{graphicx}
\usepackage{amssymb}
\usepackage{amsthm, amsmath,amssymb}
\usepackage[all]{xy}
\usepackage{graphicx}
\usepackage{psfrag}
\usepackage{multirow}
\usepackage{tikz}
\usepackage{color}

\usepackage[applemac]{inputenc}

\def\Rb{\mathbb{R}}		
\def\Cb{\mathbb{C}}  
\newcommand{\Zb}{\mathbb{Z}}
\newcommand{\Nb}{\mathbb{N}}

\newcommand{\restr}[1]
   {\vrule height1ex width.4pt
depth1.4ex\lower1.4ex\hbox{\scriptsize $\,#1$}}

\newtheorem{theorem}{Theorem}
\newtheorem{corollary}{Corollary}
\newtheorem{proposition}{Proposition}

\newtheorem{lemma}{Lemma}

\theoremstyle{definition}

\newtheorem{remark}{Remark}

\author{Inês Cruz\thanks{Centro de Matem\'atica da Universidade do Porto (CMUP), Departamento de Matem\'atica,  Faculdade de Ci\^encias da Universidade do Porto, R. Campo Alegre, 687, 4169-007 Porto, Portugal.},\,\, Helena Mena-Matos\footnotemark[1]\,\, and M. Esmeralda Sousa-Dias\thanks{Center for Mathematical Analysis, Geometry and Dynamical Systems (CAMGSD),
Departamento de Matem\'atica, Instituto Superior T\'ecnico, Av. Rovisco Pais, 1049-001 Lisboa, Portugal.}}

\begin{document}
\title{Dynamics of the birational maps arising from\\
 $F_0$ and $dP_3$ quivers }
\maketitle

\begin{abstract}
 The dynamics of the maps associated to   $F_0$ and $dP_3$ quivers is studied in  detail. We  show that  the corresponding reduced symplectic maps are conjugate to globally periodic maps by providing   explicit conjugations. The dynamics  in $\Rb^N_+$ of the original maps is obtained by lifting the dynamics of  these globally periodic maps  and the  solution of the discrete dynamical systems generated by each  map is given. 
A better understanding of the dynamics is achieved by considering first integrals. The relationship between the complete integrability of the  globally periodic maps and the dynamics of the original maps is explored.

\end{abstract}
 
\medskip

\noindent {\it MSC 2010:} Primary: 39A20; Secondary: 37E15, 37J35.\\
{\it Keywords:}   Difference equations, globally periodic map, periodic points, first integrals.

\section{Introduction}

We study the dynamics of the discrete dynamical systems (DDS)  given by the iterates of the following maps
\begin{align}
\varphi(x_1,x_2,x_3,x_4)&=\left(x_3,x_4,\frac{x_2^2+x_3^2}{x_1}, \frac{x_1^2x_4^2+ (x_2^2+x_3^2)^2}{x_1^2x_2}\right)\tag{$F_0$}\label{eq:H0}\\
\varphi(x_1,x_2,\ldots,x_6)&=\left(x_3,x_4,x_5,x_6, \frac{x_2x_4+x_3x_5}{x_1}, \frac{x_1x_4x_6+x_2x_4x_5+x_3x_5^2}{x_1x_2}\right) \tag{$dP_3$}\label{eq:dP3}
\end{align}
 in  $\Rb^4_+$ and $\Rb^6_+$, respectively. 

These maps arise in the context of  the theory of cluster algebras  associated to quivers (oriented graphs) satisfying a mutation-periodicity property. In particular, the maps \eqref{eq:H0} and \eqref{eq:dP3}  are associated to  quivers appearing in quiver gauge theories associated to the complex cones over the zero-th Hirzebruch  and the del Pezzo 3 surfaces respectively (see \cite{FengHa}, \cite{FoMa} and \cite{InEs2}). For this reason,  the quivers  in Figure~2 of \cite{InEs2}  and  their  associated maps will be called  $F_0$ and $dP_3$. These  $F_0$ and $dP_3$ quivers have 4 and 6 nodes respectively  and their mutation-period is   equal to 2. We refer to \cite{InEs2} for the construction of the maps \eqref{eq:H0} and \eqref{eq:dP3} from the respective quivers.

As shown in \cite{FoMa} any  mutation-periodic quiver   gives rise to a birational map $\varphi$ whose iterates define a system of $k$ difference equations, where $k$ is the (mutation) period of the quiver.  Moreover any quiver with $N$ nodes is represented by an $N\times N$ skew-symmetric matrix which defines a log presymplectic form $\omega$. This presymplectic form is invariant under the associated map $\varphi$ (i.e. $\varphi^*\omega  =\omega$) if and only if the quiver is mutation-periodic  \cite{InEs2}.

Our study of the dynamics of the maps \eqref{eq:H0} and \eqref{eq:dP3}  relies on a result obtained  by the first and last authors in \cite{InEs2}. Notably, any birational map  associated to a mutation-periodic quiver of period $k$ is semiconjugate to a symplectic map $\hat{\varphi}:\Rb_+^p\rightarrow \Rb_+^p$ where $p$ is the rank of the matrix  representing the quiver.  As the matrices representing the $F_0$ and $dP_3$ quivers  have rank equal  to 2,  the reduced maps  of \eqref{eq:H0} and \eqref{eq:dP3} are 2-dimensional. In  \cite{InEs2}, using one of the approaches developed there,  reduced maps  of  \eqref{eq:H0} and of \eqref{eq:dP3} preserving the symplectic form $\omega_0=\frac{dx\wedge dy}{x y}$ were obtained.

The maps \eqref{eq:H0} and \eqref{eq:dP3}  have a rather complicated expression and it seems  that  there is no simple approach to their dynamics. As the reduced maps $\hat{\varphi}$ are 2-dimensional it is then natural to begin with the  study of the dynamics of these  maps.
  We prove that the maps $\hat{\varphi}$ are globally periodic  by  exhibiting  explicit conjugations between  them and two simple globally periodic maps $\psi$ (cf. Theorem~\ref{dynred}).  As a consequence, the original maps  $\varphi$ in \eqref{eq:H0} and \eqref{eq:dP3} are semiconjugate  to the globally periodic maps $\psi$ and their dynamics  can be studied by lifting the dynamics of the respective map $\psi$.

The maps $\psi$  corresponding to  the maps  \eqref{eq:H0} and \eqref{eq:dP3}, only have a fixed point  and any other point has period $4$ and period  $6$, respectively. We show in  theorems \ref{orbits}, \ref{dymH0} and \ref{dymdP3}  that  the fixed point  of  $\psi$ lifts to  an algebraic variety invariant under $\varphi$ and that the lift of any $m$-periodic point,   with $m>1$, gives rise to $m$ algebraic varieties of codimension 2, which  are mapped cyclically one into the other by the map $\varphi$ and are invariant under the map $\varphi^{(m)}$. Moreover, Theorem~\ref{dymH0} and Theorem~\ref{dymdP3} also provide  the explicit solution of the DDS generated by the maps \eqref{eq:H0} and \eqref{eq:dP3}.   In the case of the map \eqref{eq:dP3} we are even able to  further  confine the orbits to algebraic varieties of codimension 4 by  using first integrals (cf. Proposition~\ref{prop4} and Corollary~\ref{cor1}).

The key property behind the study of the dynamics of the maps  \eqref{eq:H0} and \eqref{eq:dP3}  is the  global periodicity of  their reduced maps. It is well known that globally periodic maps are completely integrable  and   independent first integrals can be obtained  using, for instance,  the techniques described in \cite{Cima1}.  This allows us to  produce independent sets of  (lifted) first integrals of \eqref{eq:H0} and of \eqref{eq:dP3}.  It seems natural to relate the invariant varieties  already obtained to the common level sets of the lifted first integrals. This is done in Proposition~\ref{SSigma} for a particular choice of lifted first integrals of the map \eqref{eq:H0}.

The organization of the paper is as follows. In Section~\ref{sec1} we provide the necessary background and collect some results obtained in \cite{InEs2} relevant to the study to be performed. The following section is devoted to the dynamics of the reduced maps where we show that they are globally periodic. In Section~\ref{sec4} we  completely describe the dynamics of the maps \eqref{eq:H0} and \eqref{eq:dP3} by lifting the dynamics of   the corresponding globally periodic maps. The last section is devoted to the existence of first integrals  of \eqref{eq:H0} and \eqref{eq:dP3}. First,  we use first integrals to obtain a  more specific  description of the dynamics of the map \eqref{eq:dP3}. Second, we explore the 
relationship between the complete integrability of  globally periodic  maps and the   study  of the dynamics of \eqref{eq:H0} and \eqref{eq:dP3} performed in the previous sections.

\section{General setting}
\label{sec1}

Here we collect some facts and background relevant to study the dynamics of the maps $\varphi$ given by \eqref{eq:H0} and \eqref{eq:dP3}. Throughout we will consider the  dynamics of $\varphi$ on $\Rb_+^N$  (the set of points in $\Rb^N$ with positive components).
\medskip

The notion of a mutation-periodic quiver was introduced in \cite{FoMa} where it was  shown that such  a quiver gives rise to a DDS generated by  the iterates of  a  birational map $\varphi$ which is the composition of involutions (mutations) and  permutations. In particular, the maps under study are constructed (see \cite{InEs2}) from mutation-periodic quivers of period 2, namely from the  $F_0$ and $dP_3$ quivers of Figure~2 in \cite{InEs2}. One of the main results in \cite{InEs2} is the extension to quivers of arbitrary period of the reduction result  previously obtained in \cite{FoHo2} for 1-periodic quivers.  Notably, the map associated to a mutation-periodic quiver  represented by an $N\times N$ skew-symmetric matrix  of rank $p$ can be reduced to a map $\hat{\varphi}: \Rb_+^p\rightarrow \Rb_+^p$ which is symplectic with respect to a symplectic form $\omega_0$ as in Proposition~\ref{prop1} below. 
 If the rank $p$ is not maximal, then this reduction means that $\varphi$ is {\it semiconjugate} to  $\hat{\varphi}$ by a map $\Pi$. That is, there is a surjective map $\Pi:  \Rb_+^N\rightarrow  \Rb_+^p$, called a {\it semiconjugacy}, such that 
$$
\Pi\circ\varphi=\hat{\varphi}\circ\Pi.
$$

\noindent In this case the map $\varphi$ will be called a {\it lift} of $\hat{\varphi}$ by $\Pi$. 

If the matrix representing the quiver is of maximal rank, that is $p=N$, then $\Pi$ is an homeomorphism, or {\it conjugacy}, and the map $\varphi$ is said to be {\it conjugate} to $\hat{\varphi}$ (see for instance \cite{Katok} and \cite{Brin}). 

As shown in \cite{InEs2}, given a map $\varphi$ associated to a mutation-periodic quiver the reduction procedure  is constructive, enabling us to obtain a reduced map $\hat{\varphi}$ and the corresponding semiconjugacy $\Pi$. This construction can be done using either a Poisson  or a presymplectic approach. Althought the Poisson approach may fail to apply  (due to the non existence of a convenient Poisson structure) in the case of the maps \eqref{eq:H0} and \eqref{eq:dP3} both approaches are applicable. Moreover as the matrices  representing  the $F_0$ and $dP_3$ quivers   are matrices of  rank 2, the reduced maps $\hat{\varphi}$ are 2-dimensional. The next proposition summarizes the results obtained in \cite{InEs2}.

\begin{proposition}[Cruz and Sousa-Dias \cite{InEs2}]\label{prop1}
 The maps $\varphi$ in \eqref{eq:H0} or in \eqref{eq:dP3}  are semiconjugate to symplectic  maps $\hat{\varphi}: \Rb_+^2\rightarrow \Rb_+^2$. That is,  the following diagram is commutative
$$\xymatrix{ \Rb_+ ^N  \ar @{->} [r]^\varphi\ar @{->} [d]_{\Pi} & \Rb_+ ^N \ar @{->}[d]^{\Pi}\\
 \Rb_+ ^{2}  \ar @{->} [r]_{\hat{\varphi}} &  \Rb_+ ^{2}}$$
and $\hat{\varphi}^*\omega_0=\omega_0$ with $\omega_0=\frac{dx\wedge dy}{x y}$.
More precisely, 
\begin{itemize}
\item[A)]  the map  \eqref{eq:H0} is semiconjugate to $\hat{\varphi}$ by $\Pi$, with
\begin{align}\nonumber
\Pi(x_1,\ldots, x_4) &= \left(\frac{x_1x_4}{x_2^2}, \frac{x_3}{x_2}\right),\\ \label{H0red}
\hat{\varphi}(x,y)&=\left(y \left(1+\frac{(1+y^2)^2}{x^2}\right), \frac{1+y^2}{x}\right).
\end{align}
\item[B)] the map \eqref{eq:dP3} is semiconjugate to $\hat{\varphi}$ by $\Pi$, with
\begin{align}\nonumber
\Pi(x_1,\ldots, x_6) &= \left(\frac{x_2x_4}{x_3x_5}, \frac{x_1x_4x_6}{x_3x_5^2}\right),\\ \label{dP3red}
\hat{\varphi}(x,y)&=\left(\frac{y}{1+x},\frac{y (1+x+y)}{x (1+x)^2}\right).
\end{align}
\end{itemize}
\end{proposition}

We remark that the reduced maps \eqref{H0red} and  \eqref{dP3red} belong to the group of birational transformations of $\Cb^2$ preserving the symplectic form $\omega_0$ in the proposition above. This group has been studied by Blanc in \cite{Blanc} who proved that it  is generated by $SL(2, \Zb)$, the complex torus $(\Cb^*)^2$ and a (Lyness) map of order 5.
\medskip

The reduced maps $\hat{\varphi}$ in Proposition~\ref{prop1} generate the discrete dynamical systems $\mathbf{x}_{n+1}=\hat{\varphi} (\mathbf{x}_n)$, with $\mathbf{x}_n=(x_n,y_n)$. That is, the DDS generated by the  maps $\hat{\varphi}$ given by \eqref{H0red} and \eqref{dP3red} are, respectively
\begin{align}\label{disc}
\left\{\begin{array}{ll}
x_{n+1}=&\frac{y_n\left(x_n^2+ (1+y_n^2)^2\right)}{x_n^2}\\
&\\
y_{n+1}=&\frac{1+y_n^2}{x_n}
\end{array}\right. ,&\qquad\left\{\begin{array}{ll}
x_{n+1}=&\frac{y_n}{1+x_n}\\
&\\
y_{n+1}=&\frac{y_n\left(1+x_n+y_n\right)}{x_n (1+x_n)^2}
\end{array}\right. 
\end{align}
To the best of our knowledge the systems \eqref{disc} do not belong to any family of 2-dimensional systems whose dynamics has been studied before. In particular, these systems are not generated by QRT maps which is a widely studied family of integrable  maps (see for instance \cite{QRT}, \cite{Duist} and \cite{BaRo}).

\section{Dynamics of the reduced maps}

In order to study the dynamics of the  maps  \eqref{eq:H0} and \eqref{eq:dP3} it  is  natural to begin with the study of  the dynamics of the corresponding reduced maps and then lift  it to the original phase space,  since  the dynamics of  (semi)conjugate maps are related. In fact, if a map $f$ is (semi)conjugate to a map $g$ by $\Pi$, then $\Pi\circ f= g\circ \Pi$ implies
\begin{equation}
\label{conjn}
\Pi\circ f^{(n)}= g^{(n)}\circ \Pi, \quad \forall\, n\in \Nb_0,
\end{equation}
where $f^{(0)}$ is the identity map and   $f^{(m)} = f\circ \cdots\circ f$ ($m$ compositions).
 Therefore, 
the (forward) orbit of a point $\mathbf{x}$ under the map $f$  (or the  $f$-orbit through $\mathbf{x}$),
 $$\mathcal{O}_f (\mathbf{x}) = \{ f^{(n)}(\mathbf{x}): n\in\Nb_0 \},$$ 
 and the orbit of $\Pi(\mathbf{x})$ under the map $g$ satisfy
\begin{equation}\label{orbitrelation}
\Pi\left(\mathcal{O}_f (\mathbf{x})\right)=\mathcal{O}_g (\Pi(\mathbf{x})).
\end{equation}

The main result of this section is that the reduced maps \eqref{H0red} and  \eqref{dP3red}  of \eqref{eq:H0} and of  \eqref{eq:dP3} are globally periodic. 

We recall that a map $f: U\subseteq \Rb^n\rightarrow U$ is said to be {\it globally $m$-periodic} if there exists some $m\in\Nb$  such that $f^{(m)}=Id$. Hereafter we say  that $f$ is  globally $m$-periodic if $m$ is the least positive integer for which $f^{(m)}=Id$ holds. 

Difference equations defined by globally periodic maps constitute a subject of extensive research, the literature is vast and the subject has been approached  from  different points of view (see  \cite{Cima1},  \cite{Cima3}, \cite{RMM},  \cite{BK2}, \cite{HBQC},  just to refer a few). In particular, globally periodic maps  are  completely integrable and it is always possible to compute enough independent first integrals.  In Section~\ref{sec5}  we will draw some consequences of the complete integrability of the maps \eqref{H0red} and \eqref{dP3red} to the dynamics of \eqref{eq:H0} and \eqref{eq:dP3}.

\medskip

\begin{theorem}\label{dynred}
The maps $\hat{\varphi}:\Rb^2_+\rightarrow \Rb^2_+$  given by \eqref{H0red}  and  \eqref{dP3red} are globally $m$-periodic with $m=4$ and $m=6$ respectively. Moreover, 
\begin{itemize}
\item[i)] the map $\hat{\varphi}$ in \eqref{H0red} is conjugate to the map $\psi$ by the conjugacy $\widetilde{\Pi}$, with
\begin{equation}\label{conj:H0}
\psi(x,y)=\left(y,\frac{1}{x}\right), \qquad \widetilde{\Pi} (x,y)=\left(\frac{y (1+y^2)}{x}, \frac{1+y^2}{x y}\right),
\end{equation}
\item[ii)] the map $\hat{\varphi}$ in \eqref{dP3red} is conjugate to the map $\psi$ by the conjugacy $\widetilde{\Pi}$, with
\begin{equation}\label{conj:dP3}
\psi(x,y)=\left(y,\frac{y}{x}\right), \qquad \widetilde{\Pi} (x,y)=\left(x, \frac{y}{1+x}\right),
\end{equation}
\end{itemize}
where $\widetilde{\Pi}, \psi:\Rb^2_+ \rightarrow \Rb^2_+ $.
\end{theorem}
\begin{proof} Once we prove  the statements in items {\it i)} and {\it ii)} the global periodicity of the reduced  maps  $\hat{\varphi}$ follows from the invertibility of the maps $\widetilde{\Pi}$,  the global periodicity of the maps $\psi$ and the relation \eqref{conjn}. 

The maps $\psi$ given by \eqref{conj:H0} and   \eqref{conj:dP3} are globally $m$-periodic with $m=4$ and $m=6$ respectively, since
$$(x,y)\stackrel{\psi}{\longrightarrow} \left(y,\frac{1}{x}\right)\stackrel{\psi}{\longrightarrow} \left(\frac{1}{x}, \frac{1}{y}\right)\stackrel{\psi}{\longrightarrow} \left(\frac{1}{y}, x\right)\stackrel{\psi}{\longrightarrow} \left(x,y\right),
$$
and
$$(x,y)\stackrel{\psi}{\longrightarrow} \left(y,\frac{y}{x}\right)\stackrel{\psi}{\longrightarrow} \left(\frac{y}{x}, \frac{1}{x}\right)\stackrel{\psi}{\longrightarrow} \left(\frac{1}{x}, \frac{1}{y}\right)\stackrel{\psi}{\longrightarrow} \left(\frac{1}{y},\frac{x}{y}\right)\stackrel{\psi}{\longrightarrow} \left(\frac{x}{y},x\right)\stackrel{\psi}{\longrightarrow} \left(x,y\right).
$$
 The maps $\widetilde{\Pi}$  in \eqref{conj:H0} and  \eqref{conj:dP3} are invertible and their inverses are given respectively by
$$
\widetilde{\Pi}^{-1} (x,y) =\left(\frac{x+y}{y^2}\sqrt{\frac{y}{x}}, \sqrt{\frac{x}{y}}\right), \qquad \widetilde{\Pi}^{-1} (x,y) =\left(x, (1+x) y\right).
$$
 Finally,  it is easy to check that $\widetilde{\Pi}\circ\hat{\varphi}= \psi\circ\widetilde{\Pi}$ in both cases.
\end{proof}

 We note that the expressions of the reduced maps \eqref{H0red} and \eqref{dP3red} did not give any prior indication of their globally periodic behaviour. The conjugacies in the last theorem enabled us not only to realize this global periodicity property  but  also  to obtain explicit semiconjugacies between the original maps and extremely simple global periodic maps (the maps $\psi$ given in \eqref{conj:H0} and \eqref{conj:dP3}).

\bigbreak

 Each reduced map $\hat{\varphi}$ and its conjugate map $\psi$ have topologically equivalent dynamics, and the dynamics  in $\Rb^2_+$ of  the maps $\psi$ in \eqref{conj:H0} and \eqref{conj:dP3} is very simple. In fact, computing the fixed points and the periodic points of these maps one easily obtains: (a) $(1,1)$ is the unique fixed point  of both maps $\psi$,  being a center in both cases; (b) any other point is   periodic with minimal period $4$ in the case of  \eqref{conj:H0} and  minimal period $6$ in the case of \eqref{conj:dP3}.

Recall that a  periodic point $\mathbf{x}$ of $f$ is  said to have     {\it minimal period} $m$ if   $f^{(m)}(\mathbf{x})=\mathbf{x}$ and $f^{(k)}(\mathbf{x})\neq\mathbf{x}$ for all $k<m$.

Concluding, the fixed points of the DDS in \eqref{disc}  are 
$(2,1)$ and $(1,2)$ respectively, 
 and all the other points are periodic with minimal period $4$  in the case of the first system and minimal period $6$ for the  other system.

\begin{remark} We note that  not all maps arising from mutation-periodic quivers are semiconjugate to globally periodic maps. For instance, the maps which generate the Somos-4 and Somos-5 sequences  are maps associated to  quivers of mutation-period equal to 1  ($dP_1$ and $dP_2$ quivers)  and they are not semiconjugate to globally periodic maps (see \cite{FoHo} and \cite{FoHo2}).
\end{remark}

\section{Dynamics of the maps \eqref{eq:H0} and \eqref{eq:dP3}}\label{sec4}

In this section we study the dynamics of the maps \eqref{eq:H0} and \eqref{eq:dP3} from the dynamics of the  maps $\psi$ in Theorem~\ref{dynred}. The general relationship between the original maps $\varphi$, the reduced maps $\hat{\varphi}$ and their conjugate maps $\psi$ is sketched in the following commutative diagram
$$\xymatrix{ \Rb_+ ^N  \ar @{->} [r]^{\Pi}\ar @{->} [d]_{\varphi} & \Rb_+ ^2 \ar @{->}[d]^{\hat{\varphi}}\ar @{->}[r]^{\widetilde{\Pi}} &\Rb_+ ^2\ar @{->}[d]^{\psi}\\
 \Rb_+ ^{N}  \ar @{->} [r]_{\Pi} &  \Rb_+ ^{2} \ar @{->}[r]_{\widetilde{\Pi}}&\Rb_+ ^2
 }$$
 where $N=4$ for \eqref{eq:H0}, $N=6$ for \eqref{eq:dP3}, the semiconjugacies $\Pi$ and the reduced maps $\hat{\varphi}$ are given in Proposition~\ref{prop1} and the conjugacies $\widetilde{\Pi}$ and the maps $\psi$ are given in Theorem~\ref{dynred}. Note that the compositions $\pi =  \widetilde{\Pi}\circ\Pi$ provide a semiconjugacy between the maps $\varphi$ and $\psi$, since
$$ \widetilde{\Pi}\circ\Pi\circ \varphi= \psi\circ \widetilde{\Pi}\circ\Pi\Longleftrightarrow \pi\circ\varphi= \psi\circ\pi.
$$
  
  We remark that although the dynamics of the maps $\hat{\varphi}$ and $\psi$ are topologically equivalent, the same  does not hold for the dynamics of $\varphi$ and $\psi$ since $\pi=\widetilde{\Pi}\circ\Pi$ is only a semiconjugacy. However due to the global periodicity of   the maps $\psi$ we will be able to lift their dynamics to have a complete understanding of the dynamics of $\varphi$. 

\begin{theorem}\label{orbits}  Let  $f:\Rb_+^N\rightarrow\Rb_+^N$ be  semiconjugate to $g :\Rb_+^p\rightarrow\Rb_+^p$  by   $\pi:\Rb_+^N\rightarrow\Rb_+^p$ with $N>p$,  and for any $P\in \Rb_+^p$ let 
$$C_P:=\left\{\mathbf{x}\in\Rb_+^N:\,\, \pi(\mathbf{x}) =P\right\}.
$$
If $P=\pi(\mathbf{x})$ is a periodic point of $g$ of minimal period $m$ then:
\begin{enumerate}
\item $f\left(C_{g^{(i)}(P)}\right)\subset C_{g^{(i+1)}(P)}$, for all $i\in\Nb_0$. In particular, if $P$ is a fixed point of $g$ then $C_P$ is invariant under $f$.
\item $C_{g^{(i)}(P)}$ is invariant under $f^{(m)}$, for all $i\in\Nb_0$.
\item if $m>1$, the sets $C_{g^{(i)}(P)}$ with   $i=0,1,\dots,m-1$ are pairwise disjoint and  the $f$-orbit through $\mathbf{x}$ satisfies $\mathcal{O}_f (\mathbf{x})\subset \bigsqcup_{i=0}^{m-1}C_{g^{(i)}(P)}:=S_P$.
\end{enumerate}
\end{theorem}

\begin{proof}

\begin{enumerate}
\item Let $\mathbf{x}\in C_{g^{(i)} (P)}$, that is $\pi(\mathbf{x}) =g^{(i)} (P)$. By the semiconjugation assumption,   $\pi\circ f=g\circ\pi$, we have
$$\pi\left(f(\mathbf{x})\right)=g\left(g^{(i)} (P)\right)= g^{(i+1)} (P)\Longleftrightarrow f(\mathbf{x})\in C_{g^{(i+1)}(P)}.
$$
\item By item 1, $f^{(m)}\left(C_{g^{(i)}(P)}\right)\subset C_{g^{(i+m)}(P)}$ and  by the $m$-periodicity of $P$ we have $C_{g^{(i+m)}(P)}= C_{g^{(i)}(P)}$ from which the conclusion follows.
\item The assertion  that the sets $C_{g^{(i)} (P)}$ and $C_{g^{(j)} (P)}$ are pairwise disjoint  follows directly from the  assumption that  $P$ is a periodic point of $g$ with minimal period $m$. In fact,  a nonempty  intersection   would imply  that $P$ is $k$-periodic with $k<m$. 

 Finally, from \eqref{orbitrelation} we have $\pi\left(\mathcal{O}_f (\mathbf{x})\right)=\mathcal{O}_g (P)$ and by the periodicity assumption
$$\mathcal{O}_g (P)=\{P,g(P),\ldots,g^{(m-1)}(P)\}.$$ 
Therefore $\mathcal{O}_f (\mathbf{x})\subset S_P$.
\end{enumerate}
 \end{proof}
 Schematically, under the conditions of Theorem~\ref{orbits} the orbit  $\mathcal{O}_f (\mathbf{x})$ of a point  $\mathbf{x}\in C_P$, where $P$ is  an $m$-periodic point of $g$, has the following structure:

$$
   
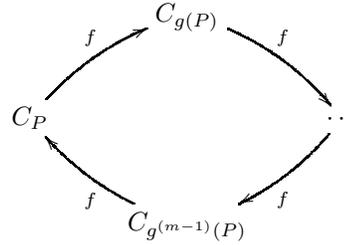
\xymatrix{ & C_{g(P)} \ar@/^/[dr]^{f} &\\
   C_P \ar@/^/[ru]^{f}&& \cdots \ar@/^/[dl]^{f}\\
               & \ar@/^/[ul]^{f}  C_{g^{(m-1)}(P)}  &&}
$$

 We remark that, as the maps  $\varphi$ in  \eqref{eq:H0} and \eqref{eq:dP3} are semiconjugate to the globally $m$-periodic maps $\psi$  in \eqref{conj:H0} and \eqref{conj:dP3} respectively, the above theorem gives the general structure  of  all the orbits  of  \eqref{eq:H0} and \eqref{eq:dP3}. More precisely, as the maps $\psi$ only have a fixed point and any other point is $m$-periodic we have:
\begin{itemize}
\item the fixed point of each $\psi$ lifts to an algebraic variety which is invariant under $\varphi$;
\item the lift of any other point gives rise to $m$ algebraic varieties of codimension 2 which  are mapped cyclically one into the other by the map $\varphi$ and are invariant under the map $\varphi^{(m)}$, where $m=4$ in the case of  \eqref{eq:H0} and $m=6$ in the case of \eqref{eq:dP3}.  In particular, the set
\begin{equation}
\label{sp}
S_P = \bigsqcup_{i=0}^{m-1}C_{\psi^{(i)}(P)}
\end{equation}
is invariant under the map $\varphi$. 
\end{itemize}

 In what follows we will use this general structure to compute the explicit solution of the DDS generated by the maps \eqref{eq:H0} and \eqref{eq:dP3}.

\subsection{Dynamics of the map  \eqref{eq:H0}}

Consider the map $\varphi$ in \eqref{eq:H0}
$$\varphi(x_1,x_2,x_3,x_4)=\left(x_3,x_4,\frac{x_2^2+x_3^2}{x_1}, \frac{x_1^2x_4^2+ (x_2^2+x_3^2)^2}{x_1^2x_2}\right).$$

From Proposition~\ref{prop1} and Theorem~\ref{dynred},
$\pi\circ\varphi=\psi\circ\pi$ with
\begin{equation}\label{H0eq1}
\psi(x,y)=\left(y,\frac{1}{x}\right),\quad \pi\left(x_1,x_2,x_3,x_4 \right)=\left(\frac{x_3(x_2^2+x_3^2)}{x_1 x_2 x_4},\frac{x_2 (x_2^2+x_3^2) }{x_1 x_3 x_4}\right).
\end{equation}
Let $P=(a,b)\in\Rb^2_+$ and $C_P=\left\{\mathbf{x}\in\Rb^4_+: \pi (\mathbf{x})=P\right\}$, that is
\begin{equation}\label{CH0}
C_{(a,b)}= \left\{(x_1,x_2,x_3,x_4)\in\Rb^4_+:\,\, x_2^2+x_3^2= \sqrt{ab}\, x_1 x_4 ,\,\, x_3= \sqrt{\frac{a}{b} }x_2\right\}.
\end{equation}

 The next theorem describes the dynamics of the map \eqref{eq:H0}.

\begin{theorem}\label{dymH0}
Let $\varphi$ be the map  \eqref{eq:H0}, $C_{(a,b)}$ given by  \eqref{CH0}  and $n\in\mathbb{N}_0$. Then, 
\begin{enumerate}
\item $\varphi (C_{(1,1)}) = C_{(1,1)}$ and for any  $\mathbf{x}=(x_1,x_2,x_3,x_4)\in C_{(1,1)}$ we have
\begin{equation}
\label{fi11}
\varphi^{(n)}(\mathbf{x})=2^{\frac{n(n-1)}{2}}\left(\frac{x_2}{x_1}\right)^n\left(x_1, 2^nx_2, 2^nx_3,4^{n}x_4 \right).
\end{equation}

\item  $\varphi^{(4)}(C_{(a,b)}) = C_{(a,b)}$ and  for any  $\mathbf{x} =(x_1,x_2,x_3,x_4)\in C_{(a,b)}$ we have
\begin{equation}
\label{fiab}
\varphi^{(4n)}(\mathbf{x})=\left (\frac{k_2}{k_1}\right )^{2 n(n-1)}\left(\frac{x_2}{x_1}\right)^{4n}\left(k_1^nx_1, k_2^nx_2, k_2^nx_3,\frac{k_2^{2n}}{k^n_1}x_4 \right),
\end{equation}
with 
\begin{equation}\label{constH0}
k_1=\frac{(1+ab)^2(a+b)^4}{a^3b^5},  \qquad k_2=\frac{(1+a b)^4 (a+b)^6}{a^5b^7}.
\end{equation}
\end{enumerate}
Moreover, $\varphi$ has no periodic points and each  component of $\varphi^{(n)}(\mathbf{x})$ goes to $+\infty$ as $n\rightarrow +\infty$.

\end{theorem}

\bigbreak
Before  proving  this theorem  we  prove a lemma which will be useful in the computations involved  in both  Theorem~\ref{dymH0} and Theorem~\ref{dymdP3}.

\begin{lemma}
\label{lemtec}
Let $g: {\cal S}\subset\Rb^n\longrightarrow  {\cal S}$ be a map of the form $g({\bf x})=G({\bf x})D{\bf x}$, with $D$ a constant diagonal matrix and $G$ a real-valued function defined on ${\cal S}$. If $G(g ({\bf x}))=c\, G({\bf x})$ for some constant $c\in\Rb$ then
  $$g^{(n)}({\bf x})=c^{\frac{n(n-1)}{2}}G^n({\bf x})D^n{\bf x},$$
for all $n\in\mathbb{N}_0$.
\end{lemma}

\begin{proof} The proof follows easily by induction on $n$. In fact, as  $G(g^{(n)}({\bf x}))=c^nG({\bf x})$ we have
\begin{align*}
g^{(n+1)}({\bf x})&=g(g^{(n)}({\bf x}))=G(g^{(n)}({\bf x}))D(g^{(n)}({\bf x}))\\
&=c^nG({\bf x})c^{\frac{n(n-1)}{2}} G^n({\bf x}) D^{n+1}{\bf x}\\
&=c^{\frac{n(n+1)}{2}} G^{n+1}({\bf x})D^{n+1}{\bf x}.
\end{align*}

\end{proof}

\begin{proof}[Proof of Theorem \ref{dymH0}]

 The map $\varphi$  is semiconjugate to the globally $4$-periodic map $\psi$ in \eqref{H0eq1} and  $(1,1)$ is the unique fixed point  of $\psi$ in $\Rb^2_+$. Therefore,  it follows from Theorem~\ref{orbits} that $C_{(1,1)}$ is invariant under $\varphi$ and $C_{(a,b)}$ is invariant under  $\varphi^{(4)}$, for all $(a,b)\in\Rb^2_+$. As  $\varphi:\Rb^4_+\rightarrow \Rb^4_+$ is invertible (it is birational) then $\varphi(C_{(1,1)})= C_{(1,1)}$ and $\varphi^{(4)}(C_{(a,b)})=C_{(a,b)}$.

The invariance of $C_{(1,1)}$ under $\varphi$ allows us to compute, for ${\bf x}\in C_{(1,1)}$, the general expression of $\varphi^{(n)}(\mathbf{x})$ through the restriction of $\varphi$ to $C_{(1,1)}$. That is, for $\mathbf{x}\in C_{(1,1)}$ we have  $\varphi^{(n)}(\mathbf{x})= \bar{\varphi}^{(n)}(\mathbf{x})$ where 
$\bar{\varphi}= \varphi\restr{C_{(1,1)}}$.

Using the expression of $\varphi$  and the definition of $C_{(1,1)}$ in \eqref{CH0} we obtain 
 $$\bar{\varphi}(\mathbf{x})=\frac{x_2}{x_1}\left(x_1,2x_2,2x_3,4x_4\right), \qquad  \mathbf{x}=(x_1,x_2,x_3,x_4)\in C_{(1,1)}.$$
The map $\bar{\varphi}$ can be written as $\bar{\varphi}(\mathbf{x})= G(\mathbf{x}) D\mathbf{x}$ with $D=\operatorname{diag} (1,2,2,4)$ and $G(\mathbf{x})=\frac{x_2}{x_1}$.   An easy computation   shows that $G(\bar{\varphi} (\mathbf{x}))= 2 G(\mathbf{x})$ and so the expression \eqref{fi11}  follows from Lemma~\ref{lemtec}. 
Any other point ${\bf x}$ lies in some $C_{(a,b)}$ which is invariant under  $\varphi^{(4)}$. Computing  $\tilde{\varphi}:= \varphi^{(4)}\restr{C_{(a,b)}}$  we obtain 
\begin{equation}\label{eq10}
\tilde{\varphi}(\mathbf{x})=\frac{x_2^4}{x_1^4}\left(k_1x_1, k_2x_2, k_2x_3,\frac{k_2^2}{k_1}x_4\right),  \quad  \mathbf{x}=(x_1,x_2,x_3,x_4)\in C_{(a,b)},
\end{equation}
with $k_1$ and $k_2$ the constants  in  \eqref{constH0}.

The  expression of $\tilde{\varphi}(\mathbf{x})$ is again of the form $\tilde{\varphi}(\mathbf{x})= G(\mathbf{x}) D\mathbf{x}$ with  $G(\mathbf{x})=\frac{x_2^4}{x_1^4}$ and 
$D=\operatorname{diag}\left( k_1,k_2, k_2, \frac{k_2^2}{k_1}\right).$
 As
$$G(\tilde{\varphi}(\mathbf{x}))= \frac{k^4_2}{k^4_1} G(\mathbf{x}),
$$
it is straightforward from   Lemma~\ref{lemtec} to obtain \eqref{fiab}  from \eqref{eq10}.

From expressions \eqref{fi11} and \eqref{fiab} it is easy to conclude that $\varphi$ has no periodic points.

The  statement that each component of  $\varphi^{(n)}$ goes to infinity when $n\rightarrow +\infty$ follows from expression \eqref{fi11} when ${\bf x}\in C_{(1,1)}$,  and from the fact that the  constants $k_1,k_2$ and $\frac{k_2}{k_1}$ in \eqref{fiab} are strictly greater than 1 for ${\bf x}\in C_{(a,b)}$  when  $(a,b)\in \Rb^2_+\setminus\{(1,1)\}$. 

\end{proof}

\subsection{Dynamics of the map  \eqref{eq:dP3}}\label{sectiondP3}

 Consider now the map $\varphi$ in \eqref{eq:dP3}:
 $$\varphi(x_1,x_2,\ldots,x_6)=\left(x_3,x_4,x_5,x_6, \frac{x_2x_4+x_3x_5}{x_1}, \frac{x_1x_4x_6+x_2x_4x_5+x_3x_5^2}{x_1x_2}\right).$$

From Proposition~\ref{prop1} and Theorem~\ref{dynred},
$\pi\circ\varphi=\psi\circ\pi$ with
\begin{equation}\label{dP3eq2}
\psi(x,y) =\left(y,\frac{y}{x}\right), \quad \pi\left(x_1,x_2,\ldots,x_6 \right) =\left(\frac{x_2x_4}{x_3 x_5},\frac{x_1x_4 x_6}{x_5 (x_2 x_4+x_3x_5)}\right).
\end{equation}
Let again $P=(a,b)\in\Rb^2_+$ and $C_P=\left\{\mathbf{x}\in\Rb^6_+: \pi (\mathbf{x})=P\right\}$, that is
\begin{equation}\label{CdP3}
C_{(a,b)}= \left\{(x_1,x_2,\ldots,x_6)\in\Rb^6_+:\,\, x_4= a \frac{x_3 x_5}{x_2},\,\, x_6= \frac{b (a+1)}{a}\frac{x_2x_5}{x_1}\right\}.
\end{equation}

\begin{theorem}\label{dymdP3}
Let $\varphi$ be the map  \eqref{eq:dP3},  $C_{(a,b)}$ given by  \eqref{CdP3} and $n\in\mathbb{N}_0$. Then, 
\begin{enumerate}
\item $\varphi (C_{(1,1)}) = C_{(1,1)}$ and for any  $\mathbf{x}=(x_1,x_2,\ldots,x_6)\in C_{(1,1)}$ we have:
\begin{itemize}
\item if $n=2m$ then 
\begin{equation}
\label{f11par}
\varphi^{(n)}(\mathbf{x}) = \lambda^m\left(x_1, 2^mx_2, 2^mx_3,4^{m}x_4 ,4^{m}x_5,8^{m}x_6\right).
\end{equation}
\item if $n=2m+1$ then 
\begin{equation}
\label{f11impar}
\varphi^{(n)}(\mathbf{x}) =(2 \lambda)^m \left(x_3,2^{m}x_4 ,2^{m}x_5,8^{m}x_6,4^m \frac{2x_3x_5}{x_1},8^{m} \frac{4x_3x_5^2}{x_1x_2}\right)
\end{equation}
\end{itemize}
with $\lambda=2^{m-1}\frac{x_5}{x_1}$. 
\item  $\varphi^{(6)}(C_{(a,b)}) = C_{(a,b)}$ and  for any  $\mathbf{x}=(x_1,x_2,\ldots,x_6)\in C_{(a,b)}$ we have:
\begin{equation}
\label{fab}
\varphi^{(6n)}(\mathbf{x})=k_1^{3n(n-1)}k_2^n\left(\frac{x_5}{x_1}\right)^{3n}\left(x_1, k_1^nx_2, k_1^nx_3,k_1^{2n}x_4,k_1^{2n}x_5, k_1^{3n}x_6\right) ,
\end{equation}
with 
\begin{equation}\label{constdP3}
k_1=\frac{(a+1)(b+1)(a+b)}{ab},  \qquad k_2=\frac{(a+1)^3 (b+1)^2(a+b)}{a^2}.
\end{equation}
\end{enumerate}
Moreover, $\varphi$ has no periodic points and each  component of $\varphi^{(n)}(\mathbf{x})$ goes to $+\infty$ as $n\rightarrow +\infty$.
\end{theorem}

\begin{proof}
The proof is done along the same lines of  the proof of the previous theorem. The map $\varphi$  is semiconjugate to the $6$-periodic map $\psi$ in \eqref{dP3eq2} and  $(1,1)$ is the unique fixed point  of $\psi$ in $\Rb^2_+$. Therefore, from items {\it 1.} and {\it 2.} of  Theorem~\ref{orbits}  it follows that $C_{(1,1)}$ is invariant under $\varphi$ and $C_{(a,b)}$ is invariant under  $\varphi^{(6)}$, for all $(a,b)\in\Rb^2_+$. The invertibility of $\varphi$ implies that $\varphi(C_{(1,1)})= C_{(1,1)}$ and $\varphi^{(6)}(C_{(a,b)})=C_{(a,b)}$.

Due to the invariance of $C_{(1,1)}$ under $\varphi$,  the general expression of $\varphi^{(n)}(\mathbf{x})$, with ${\bf x}=(x_1, x_2, \ldots, x_6)\in C_{(1,1)}$, can be computed  through the restriction of $\varphi$ to $C_{(1,1)}$. Let again $\bar{\varphi}$ denote the restriction $\varphi\restr{C_{(1,1)}}$,   that is
 \begin{equation}\label{nbar}
\bar{\varphi}(x_1, x_2, \ldots, x_6)=\left(x_3,x_4,x_5,x_6, \frac{2x_3x_5}{x_1}, \frac{4 x_3 x_5^2}{x_1 x_2}\right).
\end{equation}
 Since $\bar{\varphi}$ is not in the conditions of Lemma~\ref{lemtec},  we consider the map $\bar{\varphi}^{(2)}$ whose expression is 
$$\bar{\varphi}^{(2)}(\mathbf{x})=\frac{x_5}{x_1}\left(x_1,2x_2,2x_3,4x_4, 4x_5,8x_6\right).$$
This map can be written as $\bar{\varphi}^{(2)}(\mathbf{x})= G(\mathbf{x}) D\mathbf{x}$ with $D=\operatorname{diag} (1,2,2,4,4,8)$ and $G(\mathbf{x})=\frac{x_5}{x_1}$.  One  easily checks that $G(\bar{\varphi}^{(2)} (\mathbf{x}))= 4 G(\mathbf{x})$ and so by Lemma~\ref{lemtec}  the expression \eqref{f11par} follows.  Applying   $\bar{\varphi}$ given by \eqref{nbar} to \eqref{f11par}   the expression \eqref{f11impar}  is obtained.

Any other point ${\bf x}$ lies in some $C_{(a,b)}$, which is invariant under  $\varphi^{(6)}$. Computing   $\tilde{\varphi}:= \varphi^{(6)}\restr{C_{(a,b)}}$  we obtain
\begin{equation}\label{eq13}
\tilde{\varphi}(\mathbf{x})=k_2\frac{x_5^3}{x_1^3}\left(x_1, k_1x_2, k_1x_3,k_1^2x_4,k_1^2x_5,k_1^3x_6\right), 
\end{equation}
where $k_1$ and $k_2$ are the constants in \eqref{constdP3}.

The  expression of $\tilde{\varphi}(\mathbf{x})$ is again of the form $\tilde{\varphi}(\mathbf{x})= G(\mathbf{x}) D\mathbf{x}$ with  
$$G(\mathbf{x})=k_2\left (\frac{x_5}{x_1}\right )^3\quad   \mbox{and} \quad D=\operatorname{diag}\left( 1,k_1,k_1, k_1^2,k_1^2,k_1^3\right).$$
 As $G(\tilde{\varphi}(\mathbf{x}))= k_1^6 G(\mathbf{x})$, it is straightforward to obtain \eqref{fab} from Lemma~\ref{lemtec}.

From the expressions \eqref{f11par}, \eqref{f11impar} and \eqref{fab} it is easy to conclude that $\varphi$ has no periodic points.
The  conclusion that each component of  $\varphi^{(n)}$ goes to infinity when $n\rightarrow +\infty$ is an immediate consequence from the expressions \eqref{f11par} and \eqref{f11impar} when ${\bf x}\in C_{(1,1)}$. In the case  of ${\bf x}\in C_{(a,b)}$ with  $(a,b)\in \Rb^2_+\setminus\{(1,1)\}$ the conclusion follows from the fact that both $k_1$ and $k_2$ are   greater  than $2$ ( in fact $k_2$ is   greater than $3$).

\end{proof}

 \section{ First integrals and lifted dynamics}\label{sec5}

In the previous section we studied the dynamics of the maps \eqref{eq:H0} and \eqref{eq:dP3} by lifting the dynamics of  the globally periodic maps $\psi$  in \eqref{conj:H0} and \eqref{conj:dP3}. In particular, Theorem~\ref{orbits} guarantees that the  dynamics of those maps takes place on invariant sets $S_{P}$  defined by \eqref{sp} which are (the union of) algebraic varieties of dimension 2 in the case of \eqref{eq:H0} and of dimension 4 in the case of \eqref{eq:dP3}.

The  aim of this section is twofold: first we  obtain a better confinement (to algebraic subvarieties of dimension 2) of the orbits of the map \eqref{eq:dP3}  by finding  first integrals of  suitable restricted maps; second, we use the complete integrability of the globally periodic maps $\psi$ to produce first integrals of the maps \eqref{eq:H0} and \eqref{eq:dP3} and show how the common level sets of these first integrals relate to the invariant sets  $S_P$.

Recall that a {\it first integral} of the DDS generated by a map $f: U\subseteq \Rb^n\rightarrow U$,  or {\it a first integral of $f$}, is a non-constant real function $I$ which is constant on the orbits of the DDS.  That is,
$$
I(f(\mathbf{x}))= I(\mathbf{x}), \quad \forall\,\, \mathbf{x}\in U.
$$

The DDS generated by $f$ is called {\it completely integrable} if there are $n$ functionally independent first integrals of $f$. 

First integrals of  a map $f$  play an important role since they enable us to confine the orbits of the corresponding DDS. In fact,  if $I_1, \ldots, I_k$ are $k$ first integrals of $f$ then each orbit ${\cal O}_f({\bf x})$ stays on the common level set
$$\Sigma^{I}_{\bf c} = \{ {\bf z}\in U: I_1({\bf z}) = c_1, \ldots ,I_k({\bf z}) = c_k \}$$
with $c_j= I_j({\bf x})$, $j=1,\ldots ,k$. 

\subsection{First  integrals and the dynamics of  \eqref{eq:dP3}}

 As seen in the previous section, the dynamics of the map $\varphi$ in \eqref{eq:dP3} takes place either on $C_{(1,1)}$ or on 
$$S_P = \bigsqcup_{i=0}^{5}C_{\psi^{(i)}(P)}\qquad P\in \Rb^2_+\setminus\{(1,1)\}.$$
where $\psi (x,y) =(y, y/x)$ and $C_{P}$ is defined by \eqref{CdP3}.

$C_{(1,1)}$ is a 4-dimensional algebraic variety of $\Rb^6_+$ which is invariant under the map $\varphi$, and the orbit under $\varphi$ of a point ${\bf x}\notin C_{(1,1)}$ circulates through six pairwise disjoint algebraic varieties $C_{\psi^{(i)}(a,b)}$ of dimension 4  as schematized below. 
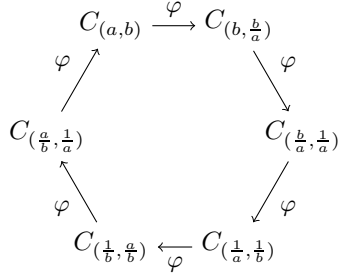
\begin{center}
\begin{tikzpicture}[scale=1][>=latex]
    \def\radius{1.7cm} 
    \node (h0A) at (60:\radius)   {$C_{(b,\frac{b}{a})}$};
    \node (h0C) at (0:\radius)    {$C_{(\frac{b}{a},\frac{1}{a})}$};
    \node (h1B) at (-60:\radius)  {$C_{(\frac{1}{a},\frac{1}{b})}$};
    \node (h1A) at (-120:\radius) {$C_{(\frac{1}{b},\frac{a}{b})}$};
    \node (h1C) at (180:\radius)  {$C_{(\frac{a}{b},\frac{1}{a})}$};
    \node (h0B) at (120:\radius)  {$C_{(a,b)}$};

    \path[->,font=\small]
        (h0A) edge node[auto] {$\varphi$} (h0C)
        (h0C) edge node[auto] {$\varphi$} (h1B)
        (h1B) edge node[auto] {$\varphi$} (h1A)
        (h1A) edge node[auto] {$\varphi$} (h1C)
        (h1C) edge node[auto] {$\varphi$} (h0B)
        (h0B) edge node[auto] {$\varphi$} (h0A);
\end{tikzpicture}

\end{center}
 Note that each $C_{\psi^{(i)}(a,b)}$ is invariant under $\varphi^{(6)}$. 

 In the next proposition we  show that the restrictions of  $\varphi$  to $C_{(1,1)}$ and of $\varphi^{(6)}$  to $C_{(a,b)}$ admit two independent first integrals. These first integrals allow us to show that  the orbits of the map \eqref{eq:dP3} are confined to algebraic varieties of dimension 2 (cf. Corollary~\ref{cor1}  below).

\begin{proposition}\label{prop4} Let $\varphi$ be the map \eqref{eq:dP3}   and $C_{(a,b)}$ the  4-dimensional algebraic  variety defined by \eqref{CdP3}. Then,
\begin{itemize}
\item[i)] the map $\bar{\varphi}=\varphi\restr{C_{(1,1)}}$ has the following  first integrals: 
\begin{equation}\label{int11}
\bar{I}_1(\mathbf{x}) = \frac{x_2}{x_3}+\frac{x_3}{x_2}  \quad \mbox{and} \quad \bar{I}_2(\mathbf{ x}) = \frac{2x_2^2}{x_1x_5}+\frac{x_1x_5}{x_2^2}.
\end{equation}
\item[ii)] each restriction $\varphi^{(6)}\restr{C_{\psi^{(i)}(a,b)}}$ has the following first integrals: 
\begin{equation}\label{intab}
\tilde{I}_1(\mathbf{ x}) = \frac{x_3}{x_2}  \quad \mbox{and} \quad \tilde{I}_2(\mathbf{ x}) = \frac{x_1x_5}{x_2^2}.
\end{equation}
\end{itemize}
\end{proposition}

\begin{proof} 
We remark that  $(x_1,x_2,x_3,x_5)$ can be taken as coordinates in any $C_{(a,b)}$. 
The expression of the restriction $\bar{\varphi} =\varphi\restr{C_{(1,1)}}$ in  these coordinates is given by
$$\bar{\varphi} (x_1, x_2, x_3, x_5)=\left(x_3, \frac{x_3x_5}{x_2}, x_5, 2\frac{x_3x_5}{x_1}\right).$$
Using this expression it is immediate to check that $\bar{I}_1$ and $\bar{I}_2$ are first integrals of $\bar{\varphi}$.

 Analogously, the expression of  $\tilde{\varphi} =\varphi^{(6)}\restr{C_{(a,b)}}$ in the same coordinates  is given by
$$\tilde{\varphi} (x_1, x_2, x_3, x_5)=k_2\frac{x_5^3}{x_1^3}\left(x_1, k_1x_2, k_1x_3, k_1^2x_5\right),$$
where $k_1$ and $k_2$ are the constants in \eqref{constdP3}.  From the above expression of $\tilde{\varphi}$ it is easy to check that  $\tilde{I}_1$ and $\tilde{I}_2$ are first integrals of this map.

A similar argument applies to the restriction of $\varphi^{(6)}$ to $C_{\psi^{(i)}(a,b)}$. One just has to replace the values of $a,b, k_1$ and $k_2$ accordingly. 
\end{proof}

 \begin{remark} We note that the integrals $\tilde{I}_1$ and $\tilde{I}_2$ in \eqref{intab} are independent  everywhere  on $ C_{(a,b)}$ and that the integrals $\bar{I}_1$ and $\bar{I}_2$ in \eqref{int11} are independent on the dense open set 
$$C_{(1,1)}\setminus\left (\left\{x_2=x_3\right\}\cup \left\{ x_1x_5 = \sqrt{2} x_2^2\right\}\right ).
$$
\end{remark}

\bigbreak

\begin{corollary}\label{cor1}  Let $\mathbf{x}$ be a point in $\Rb^6_+$ and $\mathcal{O}_\varphi (\mathbf{x})$   the orbit of $\mathbf{x}$ under the map \eqref{eq:dP3}. Then one of the following holds:
\begin{enumerate}
\item $\mathcal{O}_\varphi (\mathbf{x})$  is contained in the following  2-dimensional algebraic variety:
$$\left\{ {\bf z} \in\Rb^6_+: \,\, z_4= \frac{z_3 z_5}{z_2},\,\, z_6= 2\frac{z_2z_5}{z_1}, \,\,z_2^2 + z_3 ^2 =c z_2 z_3, \, 2z_2^4+z_1^2z_5^2= d z_1z_2^2 z_5\right\},$$
where $c, d$ are constants determined by ${\bf x}$.
\item $\mathcal{O}_\varphi (\mathbf{x})$ circulates through six pairwise disjoint  2-dimensional algebraic varieties as follows
$$D_{(a,b,c,d)}\stackrel{\varphi}{\longrightarrow} D_{h(a,b,c,d)} \stackrel{\varphi}{\longrightarrow} \cdots\cdots  \stackrel{\varphi}{\longrightarrow} D_{h^{(5)}(a,b,c,d)}\stackrel{\varphi}{\longrightarrow}D_{(a,b,c,d)}$$
where
 $$D_{(a,b,c,d)}=\left\{{\bf z}\in\Rb^6_+: z_4= a \frac{z_3 z_5}{z_2},\, z_6= \frac{b (a+1)}{a}\frac{z_2z_5}{z_1},\, z_3 =c z_2,\, z_1 z_5= d z_2^2\right\},$$
the constants $a,b,c,d$ are determined by $\mathbf{x}$  with $(a,b)\neq (1,1)$, and $h$ is the following globally 6-periodic map
\begin{equation}
\label{h}
h(\alpha,\beta,\gamma,\delta)=\left(\beta, \frac{\beta}{\alpha}, \frac{1}{\alpha\gamma}, \frac{\alpha+1}{\alpha^2\delta}\right).
\end{equation}
\end{enumerate}
\end{corollary}
\begin{proof}
 Recall that each $\mathbf{x}\in\Rb^6_+$ belongs to a set of the form $C_{(a,b)}$ as defined in \eqref{CdP3} by the equations 
$$x_4= a \frac{x_3 x_5}{x_2},\,\, x_6= \frac{b (a+1)}{a}\frac{x_2x_5}{x_1}.$$

If $(a,b)=(1,1)$ then the orbit $\mathcal{O}_\varphi (\mathbf{x})$ is contained in the variety described in {\it 1.} as a straightforward consequence of the fact that $\varphi(C_{(1,1)}) \subset C_{(1,1)}$ and of Proposition~\ref{prop4}-{\it i)}.

If $(a,b)\neq (1,1)$ then the orbit of $\mathbf{x}$ circulates through the algebraic varieties  
$$C_{(a,b)}, C_{(b,\frac{b}{a})}, C_{(\frac{b}{a},\frac{1}{a})}, C_{(\frac{1}{a},\frac{1}{b})}, C_{(\frac{1}{b},\frac{a}{b})}, C_{(\frac{a}{b},a)},$$ 
which are pairwise disjoint and invariant under $\varphi^{(6)}$. Proposition~\ref{prop4}-{\it ii)} then implies that the intersection of $\mathcal{O}_\varphi (\mathbf{x})$ with each $C_{(\alpha,\beta)}$ lies on a set of the form 
$$D_{(\alpha,\beta,\gamma,\delta)}=\{ {\bf z} \in C_{(\alpha,\beta)}: \,\, z_2 =\gamma z_3,\,  z_1 z_5=\delta z_2^2\}.$$
An easy computation shows that $\varphi(D_{(\alpha,\beta,\gamma,\delta)})\subset D_{h(\alpha,\beta,\gamma,\delta)}$ where $h$ is given by \eqref{h}. Finally note that $h$ is globally 6-periodic, as expected.

\end{proof}

\subsection{Lifted first integrals  of  \eqref{eq:H0} and \eqref{eq:dP3}}

In this subsection we explore the fact that the maps \eqref{eq:H0} and \eqref{eq:dP3} are semiconjugate to globally periodic maps in order   to: (a) obtain  a set of  two independent first integrals of \eqref{eq:H0} and of  \eqref{eq:dP3}; (b) relate the common level sets of those first integrals to the invariant sets $S_P$ described in the previous sections.

Recall that globally $m$-periodic maps are completely integrable  and it is easy to construct maximal sets of  independent first integrals of such maps, for instance by taking into account  the algebra of all symmetric polynomials in $m$ variables  as in  Cima {\it et.al.} \cite{Cima1}. Also each set of independent first integrals of a map can be lifted to a set of independent first integrals of a (semi)conjugate map. In fact,   if $f$ is (semi)conjugate to $g$ by $\pi$ (i.e. $\pi\circ f=g\circ\pi$) and $I$ is a first integral of $g$ then 
$$I\circ\pi\circ f=I\circ g\circ\pi = I\circ\pi,$$
which means  that $J=I\circ \pi$ is a first integral of $f$. We will refer to $J$ as a \emph{lifted first integral} of $f$. 

The maps  \eqref{eq:H0} and \eqref{eq:dP3} are semiconjugate to the globally periodic maps $\psi$ given in   Theorem~\ref{dynred}  by \eqref{conj:H0} and  \eqref{conj:dP3} respectively. Therefore we can compute two independent first integrals $J_1$ and $J_2$  of \eqref{eq:H0} and \eqref{eq:dP3} by lifting two independent first integrals $I_1$ and $I_2$ of the respective map $\psi$.  Consequently  the orbits under these maps are entirely contained on  common level sets  of the lifted first integrals, that is on  sets of the form
 
$$
 \Sigma_{\bf c}^J=\left\{\mathbf{z}: \, J_1(\mathbf{z})= c_1, \, J_2(\mathbf{z})= c_2\right\}.
$$

 Note that each invariant set $S_P$,  as defined  in  Theorem~\ref{orbits},  is contained in a set $\Sigma_{\bf c}^J$, for any lifted first integrals $J_1=I_1\circ\pi$ and $J_2=I_2\circ\pi$. 
 In fact, if ${\bf z}$ is a point  of $S_P$ with $P=\pi({\bf x})$,  then we have  $\pi ({\bf z})= \psi^{(i)} (P)$  for some $i\in \{0, \ldots , m-1\}$ and so
\begin{equation}\label{note}
J_k({\bf z}) = I_k (\pi ({\bf z})) = I_k (\psi^{(i)} (P)) = I_k (P) = J_k({\bf x}),
\end{equation}
 which means that ${\bf z}$ belongs to the common level set of the integrals $J_1$ and $J_2$ containing ${\bf x}$. 

Our objective now is to give a more precise description of the  relation of the invariant sets $S_P$ with $\Sigma_{\bf c}^J$ for a specific choice of lifted first integrals $J_1$ and $J_2$ of the map \eqref{eq:H0}. A similar description can be made in the case of the map \eqref{eq:dP3} (see Remark \ref{remdP3} below).
\bigbreak

\noindent{\bf Lifted first integrals of the map \eqref{eq:H0}}
\medskip

Let  $\psi$  be the  globally 4-periodic map  given in   \eqref{conj:H0},  that is
$\psi(x,y)=\left(y,\frac{1}{x}\right)$,
and let us consider the following first integrals of $\psi$:
\begin{equation}\label{integrais0}
I_1=x+y+\frac{1}{x}+\frac{1}{y}, \quad I_2=x y+\frac{1}{xy}+\frac{x}{y}+\frac{y}{x}.
\end{equation}
These first integrals are functionally independent on $\Rb_+^2\setminus L$ where $L\subset \Rb_+^2$ is the vanishing locus of the Jacobian 
\begin{equation}\label{jacob1}
\det Jac(I_1, I_2)= \frac{(x-y) (x y-1) (x^2-1) (y^2-1)}{x^3 y^3}.
\end{equation}
 
The first integrals   $I_1$ and $I_2$ lift to first integrals $J_1=I_1\circ \pi$ and $J_2=I_2\circ \pi$ of the map \eqref{eq:H0} where $\pi$ stands for the semiconjugacy  \eqref{H0eq1}. More precisely, the lifted first integrals of \eqref{eq:H0} are
\begin{equation}\label{liftH0}
J_1= \frac{x_1^2x_4^2+(x_2^2+x_3^2)^2}{x_1x_2x_3x_4}, \quad J_2= \frac{x_1^2x_4^2}{(x_3^2+x_2^2)^2}+\frac{(x_3^2+x_2^2)^2}{x_1^2x_4^2}+\frac{x_3^2}{x_2^2}+\frac{x_2^2}{x_3^2}.
\end{equation}

\bigbreak

Consequently,  the orbit ${\cal O}_\varphi({\bf{x}})$ of ${\bf x}\in\Rb^4_+$ under the map \eqref{eq:H0} stays on the algebraic variety of dimension 2  

\begin{equation}\label{novo}
\Sigma_{\bf c}^J =\{ {\bf z}\in \Rb_+^4: J_1({\bf z}) = c_1, \, J_2({\bf z}) = c_2\},
\end{equation}
where ${\bf c}=(J_1({\bf x}),J_2({\bf x}))$.

We now consider   the  invariant set described in Theorem~\ref{orbits}:
\begin{equation}\label{SPH0}
S_P= \bigsqcup_{i=0}^{3}C_{\psi^{(i)}(P)},
\end{equation}
 where $P=\pi(\mathbf{x})$ 
 and $C_Q$ is defined by \eqref{CH0}.
 As seen above $S_P$ is contained in $\Sigma_{\bf c}^J$, 
 and the next proposition provides new insights into the relationship between these two sets.
 
 \begin{proposition}\label{SSigma}
 Let $\Sigma^{J}_{\mathbf{c}}$ be the common level set \eqref{novo} of the (lifted) first integrals \eqref{liftH0} of the map \eqref{eq:H0} and  let $S_P$ be defined by \eqref{SPH0}.  Then,
 
$$\Sigma^{J}_{\mathbf{c}}=S_P\cup S_{\sigma(P)},$$
where $\sigma$  is the reflection with respect to the line $y=x$.
 
 In particular, if  $L$ is the vanishing locus of the Jacobian \eqref{jacob1}  the  following holds: 
 \begin{itemize}
 \item[i)] if $P\in L$ then $\Sigma^{J}_{\mathbf{c}}=S_P$.
 \item[ii)] if $P\notin L$ then  $\Sigma^{J}_{\mathbf{c}}=S_P\sqcup S_{\sigma(P)}$.
 \end{itemize} 
\end{proposition}

\begin{proof} 
 From \eqref{note},  $\mathbf{z}\in \Sigma^{J}_{\mathbf{c}}$ if and only if   $(u,v)=\pi({\bf z})$ satisfies $I_k(u,v)=I_k(P)$  where $I_k$ are the first integrals in \eqref{integrais0}. Considering $P=(a,b)$  this is equivalent to the following system of polynomial equations
$$\left\{\begin{array}{l}
u^2v+u v^2+u+v- \left ( a+b+\frac{1}{a}+\frac{1}{b}\right ) uv =0\\
\\
u^2v^2+u^2+v^2+1- \left ( ab+\frac{1}{ab}+\frac{a}{b}+\frac{b}{a}\right ) uv =0.
\end{array}\right. $$
It is not difficult to show that the solution set of the above system is 
$$ \left\{ (a,b), (b,a^{-1}), (a^{-1},b^{-1}), (b^{-1},a), (b,a), (a,b^{-1}), (b^{-1},a^{-1}), (a^{-1},b) \right\}.$$
This set is equal to ${\cal O}_\psi(P) \cup {\cal O}_\psi(\sigma( P))$ and so   $\Sigma^{J}_{\mathbf{c}}=S_P\cup S_{\sigma(P)}$.

Note that if $P\in L$ then ${\cal O}_\psi(P) ={\cal O}_\psi(\sigma( P))$ and consequently $\Sigma^{J}_{\mathbf{c}}=S_P$.  If $P\notin L$ then ${\cal O}_\psi(P) \cap {\cal O}_\psi(\sigma( P))=\emptyset$ and so $\Sigma^{J}_{\mathbf{c}}$ is the disjoint  union of the  sets $S_P$ and $S_{\sigma(P)}$.

\end{proof}

\begin{remark}\label{remdP3}   For the case of the map \eqref{eq:dP3} note that  the globally 6-periodic map $\psi(x,y)=\left(y,\frac{y}{x}\right)$ admits the following first integrals
\begin{equation}\label{integrais3}
I_1=x+y+\frac{1}{x}+\frac{1}{y}+\frac{y}{x}+\frac{x}{y}, \quad I_2=x^2+y^2+\frac{1}{x ^2}+\frac{1}{y ^2}+\frac{x^2}{y^2}+\frac{y^2}{x^2}.
\end{equation}
These integrals can then be lifted by the map $\pi$ in  \eqref{dP3eq2}  to first integrals $J_1=I_1\circ\pi$, $J_2=I_2\circ\pi$ of  \eqref{eq:dP3}.  The relation  in Proposition~\ref{SSigma}   between the common level sets $\Sigma^{J}_{\mathbf{c}}$ and the invariant sets $S_P$ still holds in this case.  That is, $\Sigma^{J}_{\mathbf{c}}$  coincides with either the union of two orbits or  with just  one orbit of \eqref{eq:dP3}, the latter case occurring whenever $P$ belongs to the vanishing locus of the Jacobian of the first integrals  \eqref{integrais3}. The proof is completely analogous to the proof of Proposition~\ref{SSigma}. 
\end{remark}

\noindent {\emph{Acknowledgements}.   
The work of I. Cruz and  H. Mena-Matos was partially funded by the European Regional Development
Fund through the program COMPETE and by the Portuguese Government
through FCT - Funda{\c c}{\~a}o para a Ci{\^e}ncia e a Tecnologia - 
under the project PEst-C/MAT/UI0144/2013.

The work of  M. E. Sousa-Dias  was partially funded by FCT/Portugal through project PEst-OE/EEI/LA0009/2013}.

\small{

\end{document}